\documentclass[12pt,twoside,final]{article}

\usepackage{chngpage}
\usepackage[a4paper,left=22mm,right=22mm,top=30mm,bottom=30mm]{geometry}
\usepackage{imakeidx}
\usepackage{xr-hyper}
\usepackage{xr}
\usepackage[all]{xy}
\usepackage[center]{titlesec}
\titleformat*{\section}{\large\bfseries}
\usepackage{enumerate,cite}
\usepackage[latin1]{inputenc}
\usepackage{amssymb}
\usepackage{amsthm}
\usepackage{amsmath}
\usepackage{amsfonts}
\usepackage{mathrsfs}
\usepackage{hyperref}
\usepackage{color}
\usepackage{graphicx}
\usepackage{setspace}
\usepackage{bm}
\usepackage{enumitem}
\usepackage{booktabs}

\usepackage{textcomp}
\usepackage{tikz}
\usetikzlibrary{matrix,chains}
\usepackage[center]{titlesec}
\titleformat*{\section}{\normalsize\bfseries}
\usepackage{indentfirst}

\usepackage{amscd}

\newcommand{\Aut}{\operatorname{Aut}\nolimits}
\newcommand{\Out}{\operatorname{Out}\nolimits}

\newcommand{\IBr}{\operatorname{IBr}\nolimits}

\newcommand{\Irr}{\operatorname{Irr}\nolimits}

\newcommand{\Bl}{\operatorname{Bl}}

\newcommand{\Z}{\mathrm{Z}}

\newcommand{\C}{\mathrm{C}}
\newcommand{\Ker}{\mathrm{Ker}}

\newcommand{\OO}{\mathrm{O}}

\newcommand{\tbG}{\widetilde{\mathbf{G}}}

\theoremstyle{remark}

\theoremstyle{definition}

\theoremstyle{plain}
\newtheorem{thm}{Theorem}[section]

\newtheorem{lem}[thm]{Lemma}

\newtheorem{cor}[thm]{Corollary}

\newtheorem{prop}[thm]{Proposition}

\newtheorem{order}[thm]{}

\newtheorem{conj*}{Conjecture}
\numberwithin{equation}{thm}

\usepackage{lastpage}
\usepackage{fancyhdr}

\begin{document}

\begin{center}{\Large\bf Inertial 2-blocks with abelian defect groups
}

\bigskip{\large Kun Zhang and Yuanyang
Zhou}

\bigskip{\scriptsize Faculty of Mathematics and Statistics,
Hubei University, Wuhan, China }

\smallskip{\scriptsize School of Mathematics and
Statistics, Central China Normal University,
Wuhan, 430079, P.R. China}
\end{center}

{\noindent\small{\bf Abstract} L. Puig defined inertial blocks.
In this paper, we prove that 2-blocks with defect group $C_{2^{n_1}}\times C_{2^{n_2}}\times...\times C_{2^{n_t}}$ are inertial, where $n_i\geq 2$ for all $i$.
 }


\section  {Introduction}

Let ${\cal O}$ be a complete discrete valuation ring  with  field of fractions $\cal K$ of characteristic $0$ and with
an algebraically closed residue field $k$ of characteristic $p$.
We assume that $\cal K$ is large enough for all finite groups considered below. Let $G$ be a finite group
and let $b$ be a $p$-block (or a block) of $G$
over ${\cal O}$, a central primitive idempotent of the group algebra ${\cal O}G$.
Let $P$ be a defect group of $b$ and $b_0$ the Brauer correspondent of the block $b$ in ${\rm N}_G(P)$.

A Morita equivalence between block algebras is called {\it basic} if it is induced by an
endopermutation source bimodule (see \cite{P0}).
Following \cite{P3}, the block $b$ is called {\it inertial} if the block algebras ${\cal O}Gb$ and ${\cal O}{\rm N}_G(P)b_0$ are
basically Morita equivalent.
By \cite[Corollary 1.5(c)]{H}, a block is inertial if and only if its block algebra is basically Morita equivalent to the block algebra of a block with a normal defect group.
For example, blocks of $p$-solvable groups with abelian defect groups are inertial.

2-blocks of finite quasi-simple groups with abelian defect groups has been classified (see \cite{EKK} of Theorem \ref{2-block-abel1} below).
Since then, much work on 2-blocks with abelian defect groups has been done (see \cite{EKK, EL1, Wu}).
For any positive integer $n$, let $C_{n}$ be the cyclic group of order $n$.
According to the classification,
2-blocks of finite quasi-simple groups with defect group isomorphic to $C_{2^{n_1}}\times C_{2^{n_2}}\times...\times C_{2^{n_t}}$, where $n_i\geq 2$ for all $i$, are inertial.
In this paper, we generalize this observation to all finite groups.

\begin{thm}\label{Main}
Let $G$ a finite group and let $b$ a 2-block of $G$.
If defect groups of $b$ are isomorphic to $C_{2^{n_1}}\times C_{2^{n_2}}\times...\times C_{2^{n_t}}$, where $n_i\geq 2$ for all $i$, then $b$ is inertial.
\end{thm}

\section{Reduced blocks with abelian defect groups}

Let $G$ be a finite group and $b$ a block of $G$. The block $b$ of $G$ is called {\it quasi-primitive} if every block of every normal subgroup of $G$ covered by $b$ is $G$-stable. In this case, $b$ covers a unique block of $K$ for any normal subgroup $K$ of $G$, denoted by $b_K$.
The block $b$ is called \textit{reduced} (see \cite[Proposition 6.1]{A}) if
$b$ is quasi-primitive and for any $N \unlhd G$ such that $b$ covers a nilpotent block of $N$,
$N \leq \Z(G){\rm O}_p(G)$ and
${\rm O}_{p'}(N) \leq \Z(G) \cap [G, G]$. {\rm
Denoted by ${\rm E}(G)$ the layer of $G$, by ${\rm F}(G)$ the Fitting subgroup of $G$, and by
${\rm F}^*(G)$ the generalized Fitting subgroup of $G$.

\begin{thm}\label{Reduced-block}
Let $b$ be a block of $G$ with defect group $P$.
If $b$ is reduced and $P$ is abelian, then one of the following holds.

\smallskip\noindent{\bf 1.} $P$ is normal in $G$.

\smallskip\noindent{\bf 2.}  The layer ${\rm E}(G)$ of $G$ is nontrivial and there exist normal subgroups $M$ and $L$ of $G$
    such that ${\rm F}^*(G) \leq M \leq L$, $G/L$ and $M/{\rm F}^*(G)$ are both $p'$-groups, $L = PM$ and each component of $G$
    is normal in $L$.

\end{thm}

\begin{proof}
Suppose that ${\rm E}(G)$ is trivial. We have ${\rm F}^*(G)={\rm F}(G)$. According to the assumptions, we have ${\rm O}_{p'}(G) \leq \Z(G)$, ${\rm F}^*(G) = {\rm O}_p(G)\Z(G)$, $P\subseteq\C_G({\rm F}^*(G)) \subseteq {\rm F}^*(G)$ and $P={\rm O}_p(G)$.

Suppose that ${\rm E}(G)$ is nontrivial.
Let $X_1, \cdots, X_r$ be all distinct components of $G$. For each $i$, let $d_i$ be the block of $X_i$ covered by $b_{{\rm F}^*(G)}$.
By \cite[Theorem 9.26]{N1}, $D = P \cap {\rm F}^*(G)$ is a defect group of $b_{{\rm F}^*(G)}$ and $D \cap X_i$ is a defect group of $d_i$.

We claim that ${\rm O}_p(X_i)$ is properly contained in $D \cap X_i$ for all $i$. Otherwise, there exists some $d_s$ with central defect group.
We take the maximal subset $I$ of $\{1, 2, \cdots, r\}$ such that for each $i\in I$, $d_i$ has central defect group.
Denote by ${\rm E}_1(G)$ the product of all such $X_i$ with $i \in I$.
Obviously ${\rm E}_1(G)$ is normal in $G$.
By \cite[Corollary 2.9]{MNST}, $b_{{\rm E}_1(G)}$ has central defect group.
Since $b$ is reduced, we have ${\rm E}_1(G) \leq {\rm Z}(G){\rm O}_p(G)$, which contradicts with $I$ being non-empty.

Set $K = {\rm F}^*(G)\C_G(D)$ and take a maximal $b_{{\rm F}^*(G)}$-Brauer pair $(D, e_D)$.
Since $b_{{\rm F}^*(G)}$ is $G$-invariant, by the Frattini argument, we have $G = {\rm F}^*(G){\rm N}_G(D, e_D)$, by which, we prove that $K$ is normal in $G$.
We have ${\rm O}_{p'}(K) = {\rm O}_{p'}(G) \leq \Z(G)$ and ${\rm O}_p(K) = {\rm O}_p(G) \leq \Z(G)$.
By \cite[Problem 9A.1]{I}, we have ${\rm F}^*(K) = {\rm F}^*(G)$.

We claim that every $X_i$ is normal in $K$.
For any $x \in \C_G(D)$ and any $i$, we have $D \cap X_i = (D \cap X_i)^x = D \cap X_i^x$. This forces $X_i^x = X_i$ for any $i$. Otherwise, $D \cap X_i$ has to be central in $X_i$; this contradicts with ${\rm O}_p(X_i)$ being properly contained in $D \cap X_i$. The claim is done.

The group $K$ satisfies conditions (g1) and (g3) in \cite[Definition 1.9]{D}, but the proof of \cite[Lemma 1.10(ii)]{D} does not rely on condition (g2) in \cite[Definition 1.9]{D}. Setting $S_i=X_i/Z(X_i)$, we use the proof of \cite[Lemma 1.10(ii)]{D} to prove that the conjugation in $K$ yields an embedding $$\rho:K/{\rm F}^*(G) \hookrightarrow \Out(S_1) \times \Out(S_2) \times \cdots \times \Out(S_r).$$
All $\Out(S_i)$ are solvable and thus $K/{\rm F}^*(G)$ is solvable. Obviously $P$ is a defect group of $b_K$.
Since $b$ is reduced, we replace the quotient $G/N$ in \cite[Lemma 2.4]{Ar} by $K/{\rm F}^*(G)$ and then apply the proof strategy of \cite[Lemma 2.4]{Ar} to prove that the image of $P$ in $K/{\rm F}^*(G)$ is a Sylow subgroup of $K/{\rm F}^*(G)$. By \cite[Chapter 6, Lemma 3.3]{G}, we have $K/{\rm F}^*(G)= {\rm O}_{p', p, p'}(K/{\rm F}^*(G))$.

Let $M$ be the inverse image in $K$ of ${\rm O}_{p'}(K/{\rm F}^*(G))$
and let $L$ be the inverse image in $K$ of ${\rm O}_{p', p}(K/{\rm F}^*(G))$.
Clearly $M$ and $L$ are normal in $G$.
We take $(P, e_P)$ to be a common maximal Brauer pair associated with the blocks $b$ and $b_K$.
By the Frattini argument again, we have $G = K \cdot {\rm N}_G(P, e_P)$.
Since $G/K \cong {\rm N}_G(P, e_P)/K \cap {\rm N}_G(P, e_P)$ and $\C_G(P) \leq K \cap {\rm N}_G(P, e_P)$,
$G/K$ is a $p'$-group and so is $G/L$.
Then we have $P \subseteq L$ and $P$ is a defect group of $b_L$.
Since $b_M$ is $L$-invariant, we have $L = PM$. The subgroups $L$ and $M$ are the desired subgroups in Statement 2.
\end{proof}

\begin{lem}\label{p-group-extension}
Let $G_1$ and $G_2$ be finite groups with a common normal subgroup $H$.
Let $\phi_i: G_i \to {\rm Aut}(H)$ be the homomorphism induced by the conjugation action of $G_i$ on $H$, where $i=1, 2$.
Assume that ${\rm Z}(H)$ is a $p'$-group, that the centralizer of $H$ in each $G_i$ is equal to ${\rm Z}(H)$,
that each quotient $G_i/H$ is a $p$-group, and that
$\phi_2(G_2) \leq \phi_1(G_1)$.
Then there is an injective group homomorphism from $G_2$ to $G_1$ preserving $H$ elementwise.
\end{lem}

\begin{proof}
For each $i$, let $P_i$ be a Sylow $p$-subgroup of $G_i$. By the assumptions,
we have $G_i = P_i H$ and the restriction $\phi_{P_i}$ of $\phi_i$ to $P_i$ is an isomorphism from $P_i$ to $\phi_i(P_i)$.
Since $\phi_2(G_2) \leq \phi_1(G_1)$, we adjust the choice of $P_1$ so that
$\phi_2(P_2)$ is a subgroup of $\phi_1(P_1)$.
Set $\hat{P}_2 =(\phi_{P_1}^{-1} \circ \phi_{P_2})(P_2)$. Then $\hat P_2$ is a subgroup of $P_1$. Denote by $\tau$ the correspondence from $P_2 H$ to $\hat{P}_2 H$
sending $d  h$ to $(\phi_{P_1}^{-1} \circ \phi_{P_2})(d)  h$ for any $d \in P_2$ and any $h \in H$.
Take $x\in Q=P_2\cap H$.
We have $\phi_1^{-1} (\phi_2(x))\in xZ(H)$.
Since $Z(H)$ is a $p'$-group, we have $\phi_1^{-1} (\phi_2(x))=x$. So
$\tau$ is well-defined. Now it is trivial to prove that $\tau$ is a group isomorphism preserving $H$ elementwise.
\end{proof}

\begin{prop}\label{Inertial-PE(G)}
Let $G$ be a finite group, $X_1, \cdots, X_n$ all components of $G$ and $b$ a block of $G$ with abelian defect group $P$.
Assume that each $X_i$ is normal in $G$, that $G = PX_1 X_2 \cdots X_n$, and that the unique block of $PX_i$ covering $c_i$ is inertial,
where $c_i$ is the block of $X_i$ covered by $b$. Then the block $b$ is inertial.
\end{prop}

\begin{proof}
Clearly we have $E(G)=X_1\cdots X_n$. Let $c$ be the block of ${\rm E}(G)$ covered by $b$. Then $P$ stabilizes $c$ and $b$ is equal to $c$. By \cite[Lemma 2.1]{Wu}, $c$ can be uniquely written as the product $d_1d_2\cdots d_n$, where $d_i$ is some block of $X_i$ for each $i$. The uniqueness implies that $P$ stabilizes each $d_i$.
Each $d_i$ is covered by $c$ and thus covered by $b$. We have $c_i=d_i$ for all $i$.

Since $P$ is abelian, ${\rm O}_p(G)$ is central in $G$.
By \cite[Corollary 1.14]{P1}, we reduce the proof of the lemma to the case ${\rm O}_p(G) = 1$.
In this case, ${\rm O}_p(X_i)=1$ for all $i$.

Set $G_i = PX_i$ and $\bar{G}_i = G_i/{\rm O}_p(G_i)$ and let $\bar{P}_i$ be the image of $P$ in $\bar{G}_i$. The inclusion $X_i\subseteq G_i$ induces an injective group homomorphism $X_i\rightarrow \bar G_i$ and we may identify each $X_i$ as a normal subgroup of $\bar{G}_i$. Since $P$ stabilizes $c_i$, $\bar P_i$ stabilizes $c_i$ and thus each $c_i$ is a block of $\bar G_i$ with defect group $\bar P_i$. Similarly, each block $c_i$ is a block of $G_i$ with defect group $P$.
Since $P$ is abelian, the group ${\rm O}_p(G_i)$ is central in $G_i$. The natural homomorphism $G_i\rightarrow \bar G_i$ induces an algebra homomorphism ${\cal O}G_i\rightarrow {\cal O}\bar G_i$, which sends $c_i$ onto $c_i$. By \cite[Corollary 1.14]{P1}, each $\bar c_i$ is an inertial block of $\bar G_i$ with defect group $\bar P_i$.

Set $\tilde{G}=\bar{G}_1 \times \bar{G}_2 \times \cdots \times \bar{G}_n$
and $\tilde{X}=X_1 \times X_2 \times \cdots \times X_n$.
Let $\tau$ be the canonical homomorphism
from $\tilde{X}$ to ${\rm E}(G)$ sending $(x_1, x_2, \cdots, x_n)$ to $x_1x_2 \cdots x_n$ for any $x_i \in X_i$.
Let $\tilde{c}$ be the block of $\tilde{X}$ such that $\tau(\tilde{c}) = c$. We identify ${\cal O}\tilde G$ and ${\cal O}\bar G_1\otimes_{\cal O}\cdots\otimes_{\cal O}{\cal O}\bar G_n$ through the obvious isomorphism ${\cal O}\tilde G\cong {\cal O}\bar G_1\otimes_{\cal O}\cdots\otimes_{\cal O}{\cal O}\bar G_n$. Since $c=c_1\cdots c_n$,
by \cite[Lemma 2.1]{Wu}, we easily conclude that $\tilde{c}$ is equal to the tensor product $c_1 \otimes c_2 \otimes \cdots \otimes c_n$.
Since each $c_i$ is an inertial block of $\bar G_i$ with defect group $\bar P_i$,
$\tilde{c}$ is an inertial block of $\tilde G$ with defect group $\bar P_1\times\cdots\times \bar P_n$.

Set $\hat{G}=\tilde{G}/{\rm Ker}(\tau)$. Clearly ${\rm E}(\hat{G})=\tilde X/{\rm Ker}(\tau)$ and we
identify ${\rm E}(\hat{G})$ with ${\rm E}(G)$ through $\tau$.
In particular, $G$ and $\hat{G}$ both contain ${\rm E}(G)$ as a normal subgroup. Applying Lemma \ref{p-group-extension} to finite groups $G$ and $\hat G$ and their common subgroup ${\rm E}(G)$, we get an injective group homomorphism from $G$ to $\hat G$ preserving ${\rm E}(G)$ elementwise. We regard $G$ with a subgroup of $\hat G$.
Since $\tau(\tilde c)=c$ and $\tilde{c}$ is an inertial block of $\tilde G$, $c$ is an inertial block of $\hat G$. Note that $c$ is also a block of $G$. By \cite[Theorem 3.13]{P2}, $c$ is inertial as a block of $G$.
\end{proof}

Let $G$ be a finite group and $b$ a block of $G$. Assume that $b$ is reduced, that a defect group $P$ of $b$ is abelian, and that the layer ${\rm E}(G)$ of $G$ is nontrivial.
Let $X_1, \cdots, X_n$ be all components of $G$, $c$ the block of ${\rm E}(G)$ covered by $b$, and for each $i$, let $c_i$ be the block of $X_i$ covered by $c$.
By Theorem \ref{Reduced-block}(ii), each $X_i$ is stable under the $P$-conjugation.
Since $b$ is reduced, $c$ is $P$-invariant and it can be viewed as a block of $P{\rm E}(G)$ with defect group $P$. As we see in the first paragraph of Proposition \ref{Inertial-PE(G)},
$c$ can be uniquely written as the product $c_1\cdots c_n$, each $c_i$ is $P$-invariant and thus can be viewed as a block of $G_i= PX_i$ with defect group $P$.

\begin{cor}\label{Compenent-inertial}
Keep the notation and the assumptions as above. Assume that each $c_i$ as a block of $G_i$ is inertial. Then
the block $b$ is inertial.
\end{cor}

\begin{proof}
By Theorem \ref{Reduced-block}(ii),
there exist normal subgroups $M$ and $L$ of $G$
such that ${\rm F}^*(G) \leq M \leq L$, $G/L$ and $M/{\rm F}^*(G)$ are both $p'$-groups, $L = PM$ and each component of $G$
is normal in $L$. Let $b_L$ be the block of $L$ covered by $b$.
Since each $c_i$ is inertial as a block of $G_i$, by Lemma \ref{Inertial-PE(G)}, $c$ is inertial as a block of $P{\rm E}(G)$. Since the block $b$ is reduced, we have
${\rm F}^*(G) = {\rm O}_p(G)\Z(G){\rm E}(G)$. Clearly, ${\rm O}_p(G){\rm E}(G)$ is a normal subgroup of $M$ with index coprime to $p$.
By \cite[Corollary 1.13]{ZZ}, $b_L$ is inertial. Since $G/L$ is a $p'$-group, by \cite[Corollary]{Z}, $b$ is inertial.
\end{proof}

\section{Proof of Theorem \ref{Main}}

In this section, we always assume $p = 2$ and give a proof of Theorem \ref{Main}.
We firstly recall the classification of blocks of quasi-simple groups with abelian defect groups.

\begin{thm}[\!\!{\cite[Theorem 6.1]{EKK}}]\label{2-block-abel1}
Let $H$ be a finite quasi-simple group. If $c$ is a block of $H$ with abelian defect group $D$, then one (or more) of the following holds.
\begin{enumerate}[itemsep=-3pt,topsep=7pt,label=\emph{(\roman*)}]
    \item $H/Z(H)$ is one of $A_1(2^a)$, ${}^2G_2(q)$ (where $q \geq 27$ is a power of $3$ with odd exponent),
    or ${\rm J}_1$, $c$ is the principal block, and $D$ is elementary abelian.
    \item $H$ is ${\rm Co}_3$, $c$ is a non-principal block, and $D \cong C_2 \times C_2 \times C_2$ (there is one such block).
    \item $c$ is Morita equivalent to a block $d$ of $\mathcal{O}L$, where $L = L_0 \times L_1 \leq H$ such that $L_0$ is
    abelian and the block of $L_1$ covered by $d$ has Klein 4-defect groups.
    \item $c$ is nilpotent covered.
\end{enumerate}
\end{thm}

\begin{lem}\label{Hy-H}
Let $H$ be a finite quasi-simple group and $c$ a block of $H$ with abelian defect group.
If a hyperfocal subgroup of the block $b$ is isomorphic to $C_{2^{m_1}} \times C_{2^{m_2}} \times \cdots \times C_{2^{m_r}}$, where $m_i \geq 2$ for all $i$,
then $c$ is nilpotent covered.
\end{lem}

\begin{proof}Let $(D, f)$ be a maximal $c$-Brauer pair and $Q$ the hyperfocal subgroup of the block $c$ with respect to $(D, f)$.
In both Case (i) and Case (ii) of Theorem \ref{2-block-abel1}, the hyperfocal subgroup $Q$
is elementary abelian.
In Case (iii) of Theorem \ref{2-block-abel1},
by \cite[Proposition 2.7(iii)]{EL1}, $Q$ is either Klein four group or trivial.
Since $Q$ is isomorphic to $C_{2^{m_1}} \times C_{2^{m_2}} \times \cdots \times C_{2^{m_r}}$, where $m_i \geq 2$ for all $i$, by Theorem \ref{2-block-abel1},
$c$ is nilpotent covered.
\end{proof}

\begin{prop}\label{Hy-PH}
Let $H$ be a finite quasi-simple group, $\widetilde{H}$ a finite group containing $H$ as a normal subgroup,
and $\tilde{c}$ a block of $\widetilde{H}$ with abelian defect group $P$.
Assume that $\widetilde{H} = PH$ and $P$ is isomorphic to $C_{2^{n_1}} \times C_{2^{n_2}} \times \cdots \times C_{2^{n_t}}$, where $n_i \geq 2$ for all $i$.
Then $\tilde{c}$ is inertial.
\end{prop}

We divide the proof of Proposition \ref{Hy-PH} into Lemma \ref{Nil-cov}--\ref{Lie-iner}.

\begin{lem}\label{Nil-cov}
Let $c$ be the block of $H$ covered by $\tilde{c}$.
Then $c$ is nilpotent covered.
\end{lem}

\begin{proof}
Let $(P,e)$ be a maximal $\tilde{c}$-Brauer pair  and let $Q$ be the hyperfocal subgroup of $\tilde{c}$ with respect to $(P,e)$.
If $Q=1$, then the block $\tilde{c}$ is nilpotent.
Assume that $Q$ is non-trivial.
By \cite[Theorem 1(iii)]{W2}, we have $P=Q\times R$  where $R=\C_P({\rm N}_{\widetilde{H}}(P,e))$.
By our assumption on $P$, $Q$ is isomorphic to  $C_{2^{m_1}}\times C_{2^{m_2}}\times...\times C_{2^{m_r}}$ where $1\leq r\leq t$ and $m_i\geq 2$ for all $i$.
By \cite[Proposition 2.5]{EL1}, $Q$ is also a hyperfocal subgroup of $c$.
By Lemma \ref{Hy-H}, $c$ is nilpotent covered.
\end{proof}

\begin{lem}\label{Spo-sim}
Assume that $H/Z(H)$ is a sporadic group, or the Tits group ${}^2F_4(2)^{\prime}$, or an Alternating group,
  or a simple group of Lie type with the defining characteristic 2.
Then $\tilde{c}$ is nilpotent.
\end{lem}

\begin{proof}
Let $c$ be the block of $H$ covered by $\tilde{c}$. By
Lemma \ref{Nil-cov}, the set of irreducible Brauer characters ${\rm IBr}(c)$ is a single ${\rm Aut}(H)_c$-orbit,
where ${\rm Aut}(H)_c$ is the stabilizer of $c$ in ${\rm Aut}(H)$.
By \cite[Theorem 2.8 and Proposition 2.6 and 3.1]{MNS}, $c$ is nilpotent.
Then $\tilde{c}$ is nilpotent too.
\end{proof}

Since we are assuming $p=2$ and ${\rm IBr}(c)$ is a single ${\rm Aut}(H)_c$-orbit (see \ref{Nil-cov}), by \cite[Proposition 2.11]{MNS}, a finite group $H$ of Lie type has no exceptional covering group with faithful block with non-central defect group. So,
in order to prove Proposition \ref{Hy-PH}, by Lemma \ref{Spo-sim}, we may assume that
there exists a simple linear algebraic group $\mathbf{G}$ of simply connected type over an algebraically closed field of odd characteristic with a Steinberg endomorphism $F: \mathbf{G} \rightarrow \mathbf{G}$
such that, denoting by $G$ the group $\mathbf{G}^F$ of fixed points, $H$ is equal to $G/Z$ for some central subgroup $Z$ of $G$.
Let $b$ be the block of $G$ dominating $c$.
By Lemma \ref{Nil-cov}, $\IBr(c)$ is a single $\Aut(H)_{c}$-orbit.
By \cite[Corollary B.8]{N2},
we have $\Aut(H) = \Aut(G)_Z$.
So $\IBr(b)$ is a single $\Aut(G)_{b}$-orbit too.

Let $(\mathbf{G}^*, F^*)$ be a dual pair of $(\mathbf{G}, F)$ (see \cite[P.118]{CE}).
By \cite[Theorem 9.12]{CE}, there exists a semisimple $2'$-element $s$ in $\mathbf{G}^{*F^*}$
such that the intersection $\Irr(b) \cap \mathcal{E}(G, s)$ is nonempty, where $\mathcal{E}(G, s)$ denotes the rational series associated to $s$.
Let $\iota: \mathbf{G} \hookrightarrow \widetilde{\mathbf{G}}$ be a regular embedding,
and let $\iota^*$ be the dual map $\widetilde{\mathbf{G}}^* \rightarrow \mathbf{G}^*$ induced by $\iota$ (see \cite[Section 15.1]{CE}). Identifying $\mathbf{G}$ with a subgroup of $\widetilde{\mathbf{G}}$ via $\iota$, we have $\widetilde{\mathbf{G}} = \mathbf{G} \Z(\widetilde{\mathbf{G}})$.
The Steinberg endomorphism $F$ on $\mathbf{G}$ can be extended to a Steinberg endomorphism on $\widetilde{\mathbf{G}}$, denoted still by $F$. Set $\widetilde{G} = \widetilde{\mathbf{G}}^F$.
By \cite[Lemma 1.7.12 and Proposition 15.6]{CE}, there exists a semisimple $2'$-element $\tilde{s}$ of $\widetilde{\mathbf{G}}^{*F^*}$ such that $s = \iota^*(\tilde{s})$ and such that
there exists a block $\tilde{b}$ of $\widetilde{G}$ covering $b$
such that
the intersection $\Irr(\tilde{b}) \cap \mathcal{E}(\widetilde{G}, \tilde{s})$ is non-empty.

\begin{lem}\label{Lie-nil}
Keep the notation and the assumptions as above.
Then either $b$ or $\tilde{b}$ is nilpotent.
\end{lem}

\begin{proof}
We apply \cite[Proposition 5.3]{EKK} to prove this lemma. We remind that \cite[Chapter 5]{EKK} retains the assumptions of \cite[Chapter 4]{EKK}, which excludes the Suzuki or Ree groups.
Suppose that $H/\Z(H) \cong {}^2G_2(3^{2n+1})$ for some $n \geq 1$. Then $H = H/\Z(H)$ and Sylow 2-subgroups of $H$ are
elementary abelian of order $8$. According to the proof of Lemma \ref{Nil-cov}, $H$ has a 2-subgroup $Q$ isomorphic to  $C_{2^{m_1}}\times C_{2^{m_2}}\times...\times C_{2^{m_r}}$ where $1\leq r\leq t$ and $m_i\geq 2$ for all $i$. This causes a contradiction.
Therefore, $H/\Z(H)$ is neither a Suzuki group nor a Ree group.

Assume that $\C_{\tbG^{*}}(\tilde{s})$ is a torus. By \cite[Lemma 4.2]{EKK}, $\tilde{b}$ is nilpotent in this case.

Assume that $\C_{\tbG^{*}}(\tilde{s})$ is not a torus. Then $\C^{\circ}_{\bf{G}^{*}}(s)$ is not a torus and now we are in the situation of \cite[Proposition 5.3]{EKK}.
By the last paragraph of the proof of \cite[Theorem 6.1]{EKK}, either $b$ is nilpotent or $c$ is  as stated  in Theorem \ref{2-block-abel1}(iii).
Suppose that $c$ is non-nilpotent and as stated  in  Theorem \ref{2-block-abel1}(iii). As we see in the proof of Lemma \ref{Hy-H}, the hyperfocal subgroup $Q$ of $c$ is a Klein four group,
which, by \cite[Proposition 2.5]{EL1}, is also a hyperfocal subgroup of $\tilde{c}$.
But by \cite[Theorem 1(iii)]{W2}, the defect group $P$ of $\tilde{c}$ is isomorphic to a direct product of $Q$ and an abelian $2$-group. This
contradicts the assumption on $P$ in Proposition \ref{Hy-PH}.
So $c$ is nilpotent and so is $b$.
\end{proof}


\begin{lem}\label{Lie-iner}
Keep the notation and the assumptions as above.
If $\tilde{b}$ is nilpotent, then $\tilde{c}$ is inertial.
\end{lem}

\begin{proof} By \cite[Theorem 2.5.1]{GLS}, $\Aut(G)$ is induced by $\widetilde{G} \rtimes \mathcal{D}$, where $\mathcal{D}$ denotes the group generated by suitable
graph and field automorphisms of $G$. Since $H=G/Z$, by \cite[Corollary B.8]{N2},
$\Aut(H)$ is induced by $\widetilde{G} \rtimes \mathcal{D}_{Z}$, where $\mathcal{D}_Z$ denotes the stabilizer of $Z$ in $\mathcal{D}$.

The kernel ${\rm Ker}(c)$ of the block $c$ is a normal $p'$-subgroup of $H$ and it is contained in the center of $H$. Denote by $\bar H$ the quotient group $H/{\rm Ker}(c)$ and by $\bar c$ the image of $c$ in the group algebra ${\cal O} \bar H$. The natural homomorphism $H\rightarrow \bar H$ induces an algebra isomorphism ${\cal O} H c\cong {\cal O}\bar H\bar c$.
Since $\tilde c$ covers $c$ and $\tilde H=PH$, we have ${\rm Ker}(c)= {\rm Ker}(\tilde c)$. In order to prove this lemma, we may assume $\Ker(\tilde{c}) = \Ker(c) = 1$. Then, since the block $b$ dominates the block $c$, we have $\Ker(b)=\OO_{2'}(Z)$ and
$c$ is the image of $b$ in the group algebra ${\cal O} H$.

Set
$
\widetilde{\mathcal{G}} = \widetilde{G}/(Z\OO_2(\widetilde{G}))$,
$\widetilde{\mathcal{A}} = (\widetilde{G} \rtimes \mathcal{D}_{Z})/(Z\OO_2(\widetilde{G}))$ and
$\widetilde{\mathcal{H}} = \widetilde{H}/\OO_2(H)$.
Let $\mathcal{H}$ and $\mathcal{P}$ be the respective images of $H$ and $P$ in $\widetilde{\mathcal{H}}$.
We have $\OO_2(\widetilde{G}) \cap G = \OO_2(G)$, $\OO_2(H)=\OO_2(G)Z/Z$ and isomorphisms
\[
\mathcal{H} \cong H/\OO_2(H) \cong G/(Z\OO_2(G)) \cong (G \OO_2(\widetilde{G}))/(Z\OO_2(\widetilde{G})).
\]
Through these isomorphisms, we identify $\mathcal{H}$ with a subgroup of $\widetilde{\mathcal{A}}$, which is normal in $\widetilde{\mathcal{A}}$.
Let $\mathcal{Q}$ be a Sylow $2$-subgroup of $\widetilde{\mathcal{A}}$ and set $\widetilde{\mathcal{B}} = \mathcal{Q}\mathcal{H}$.
By Lemma \ref{p-group-extension}, we may identify $\widetilde{\mathcal{H}} = \mathcal{P}\mathcal{H}$ with a subgroup of $\widetilde{\mathcal{B}}$.
Since $\widetilde{G}/G$ is abelian (see \cite[\S 15.1]{CE}), so is $\widetilde{\mathcal{G}}/\mathcal{H}$.
Let $\widetilde{\mathcal{K}}$ be the largest normal subgroup of $\widetilde{\mathcal{G}}$ containing $\mathcal{H}$ such that $\widetilde{\mathcal{K}}/\mathcal{H}$ is a $2'$-group and $\widetilde{\mathcal{G}}/\widetilde{\mathcal{K}}$ is a $2$-group.

Let $\tilde{b}'$ be the block of $\widetilde{\mathcal{G}}$ dominated by $\tilde{b}$
and $c'$ the block of $\mathcal{H}$ dominated by $c$.
Since $\tilde{b}$ covers $b$ and $b$ dominates $c$, $\tilde{b}'$ covers $c'$.
Let $\tilde{d}'$ be a block of $\widetilde{\mathcal{K}}$ that covers $c'$ and is covered by $\tilde{b}'$.
Since $\tilde{b}$ is nilpotent, $\tilde{b}'$ is nilpotent and so is $\tilde{d}'$.
Let $\Bl(\widetilde{\mathcal{K}} | c')$ be the set of blocks of $\widetilde{\mathcal{K}}$ covering $c'$.
Since $\widetilde{\mathcal{K}}/\mathcal{H}$ is abelian, by \cite[Lemma 2.2]{KM}, the cardinality of $\Bl(\widetilde{\mathcal{K}} | c')$ is odd
and every block within $\Bl(\widetilde{\mathcal{K}} | c')$ is nilpotent.

The construction of $\widetilde{\mathcal{K}}$ ensures its stability under the action of $\mathcal{P}$.
In particular, $\widetilde{\mathcal{K}}$ is normal in $\mathcal{P}\widetilde{\mathcal{K}}$.
Let $\tilde{c}'$ be the block of $\widetilde{\mathcal{H}}$ dominated by $\tilde{c}$.
Obviously $\mathcal{P}$ is a defect group of $\tilde{c}'$, $\tilde{c}'$ covers $c'$, $\mathcal{P}$ stabilizes $c'$ and the set $\Bl(\widetilde{\mathcal{K}} | c')$ is
$\mathcal{P}$-invariant.
There exists a $\mathcal{P}$-invariant block $\tilde{f}'$ of $\widetilde{\mathcal{K}}$ that covers $c'$.
By \cite[Problems 9.4]{N1},
$\mathcal{P}$ is a defect group of $\tilde{f}'$ as a block of $\mathcal{P}\widetilde{\mathcal{K}}$.

Let $\mathcal{L} = \mathcal{H}{\rm N}_{\mathcal{P}\widetilde{\mathcal{K}}}(\mathcal{P})$ and $e'$ the block of $\mathcal{L}$ such that the induced block $(e')^{\mathcal{P}\widetilde{\mathcal{K}}}$ is equal to $\tilde{f}'$.
Since $\mathcal{P}$ is abelian, the Brauer category of $\tilde{f}'$ as a block of $\mathcal{P}\widetilde{\mathcal{K}}$ is controlled by ${\rm N}_{\mathcal{P}\widetilde{\mathcal{K}}}(\mathcal{P})$.
Since $\tilde{f}'$ is nilpotent as a block of $\mathcal{P}\widetilde{\mathcal{K}}$,
$e'$ is nilpotent.
Clearly, $e'$ covers $c'$, and $\tilde{c}'$ is the unique block of $\widetilde{\mathcal{H}}$ covering $c'$.
Since $\widetilde{\mathcal{H}}$ is normal in $\mathcal{L}$, $e'$ covers $\tilde{c}'$.
By \cite[Theorem 3.13]{P2}, the block $\tilde{c}'$ is inertial and so is $\tilde{c}$.
\end{proof}

\begin{order}\label{proof}{Proof of Theorem \ref{Main}}.
{\rm
Let $G$ be a finite group and $b$ a $2$-block of $G$.
Assume that defect groups of $b$ are isomorphic to $C_{2^{n_1}} \times C_{2^{n_2}} \times \cdots \times C_{2^{n_t}}$, where $t \geq 1$ and $n_i \geq 2$ for all $i$.
By \cite[Proposition 6.1]{A}, we assume that $b$ is reduced.
By Corollary \ref{Compenent-inertial} and Proposition \ref{Hy-PH}, we conclude that $b$ is inertial. }
\end{order}


\end{document}